\documentclass[10pt]{amsart}
\usepackage{amsmath,amssymb,amsthm,amsfonts,enumerate,hyperref}
\theoremstyle{plain}
\newtheorem{theorem}{Theorem}[section]
\newtheorem{proposition}[theorem]{Proposition}

\newtheorem{lemma}[theorem]{Lemma}

\theoremstyle{definition}

\theoremstyle{remark}
\newtheorem{remark}[theorem]{Remark}

\numberwithin{equation}{section}

\newcommand{\Ric}{\mathrm{Ric}}

\begin{document}

\title{Conformally K\"ahler Geometry and Quasi-Einstein metrics}
\begin{abstract}
We prove that the quasi-Einstein metrics found by L\"u, Page and Pope on $\mathbb{C}P^{1}$-bundles over Fano K\"ahler-Einstein bases are conformally K\"ahler and that the K\"ahler class of the conformal metric is a multiple of the first Chern class. A detailed study of the lowest-dimensional example of such metrics on $\mathbb{C}P^{2}\sharp \overline{\mathbb{C}P}^{2}$ using the methods developed by Abreu and Guillemin  for studying toric K\"ahler metrics is given. Our  methods  yield, in  a unified framework, proofs of the existence of the Page, Koiso-Cao and L\"u-Page-Pope  metrics on $\mathbb{C}P^{2}\sharp \overline{\mathbb{C}P}^{2}$.   Finally, we investigate the properties that similar quasi-Einstein metrics would have if they also exist on the toric surface $\mathbb{C}P^{2}\sharp 2 \overline{\mathbb{C}P}^{2}$. 
\end{abstract}
\author{Wafaa Batat}
\address{Ecole Nationale Polytechnique d'Oran, B.P 1523 El M'naouar, 31000 Oran, Algeria}
\email{batatwafa@yahoo.fr}
\author{Stuart J. Hall}
\address{Department of Applied Computing, University of Buckingham, Hunter St., Buckingham, MK18 1G, U.K.} 
\email{stuart.hall@buckingham.ac.uk}
\author{Ali Jizany}
\address{Department of Applied Computing, University of Buckingham, Hunter St., Buckingham, MK18 1G, U.K.} 
\email{ali.jizany@buckingham.ac.uk}
\author{Thomas Murphy}
\address{Department of Mathematics, California State University Fullerton, 800 N. State College Blvd., Fullerton, CA 92831, USA.}
\email{tmurphy@fullerton.edu}
\maketitle
\section{Introduction}
A quasi-Einstein metric is a complete Riemannian manifold $(M,g)$, satisfying 
\begin{equation}\label{QEMeq}
Ric(g)+\nabla^{2}\phi-\frac{1}{m}d\phi\otimes d\phi=\lambda g,
\end{equation}
for some function $\phi \in C^{\infty}(M)$ and constants $\lambda$, $m$ with $m>0$.  Setting $\phi$ to be constant yields an Einstein metric,  so solutions to Equation (\ref{QEMeq}) with  nonconstant $\phi$  are referred to as \textit{non-trivial} quasi-Einstein metrics. When $m$ is a positive integer such metrics are important as the base manifolds for warped product constructions of Einstein metrics.  As well as generalising the Einstein condition, quasi-Einstein metrics can be thought of as deformations of gradient Ricci solitons, which are of central importance in the theory of  Ricci flow. By formally taking $m\rightarrow \infty$ in Equation (\ref{QEMeq})  one recovers the equation defining a gradient Ricci soliton. Given their relationship with both types of canonical metric, a fundamental question is: in what way are quasi-Einstein metrics like Ricci solitions, and in what way are they like Einstein metrics?\\
\\
The only known examples of compact quasi-Einstein metrics where $m$ varies continuously are essentially due to a construction of L\"u, Page and Pope \cite{LPP} on $\mathbb{C}P^1$-bundles over Fano K\"ahler-Einstein manifolds,  where the total space is denoted  $W_q$. This construction was generalised by the second author in \cite{Hallqem}. On such spaces, non-K\"ahler Einstein metrics were known to exist due to a similar construction of B\'erard-Bergery \cite{BB}  (later generalised by Wang and Wang \cite{WW}).  These spaces also admit shrinking gradient Ricci solitons due to Koiso \cite{Koi}, Cao \cite{Cao} and Chave and Valent \cite{CV}. A more general construction of such solitons was later given by  Dancer and Wang \cite{DW}.  All the examples of Ricci solitons on these manifolds are K\"ahler. The B\'erard-Bergery Einstein metrics turn out to be conformally K\"ahler. However quasi-Einstein metrics are \emph{never} K\"ahler, due to a foundational result of Case, Shu and Wei \cite{CSW}. Nevertheless, we will show K\"ahler geometry plays a role in the theory of quasi-Einstein metrics on these spaces. \\
\\
In Section \ref{Sect2} we show that the L\"u-Page-Pope metrics are conformally K\"ahler and so are similar to the B\'erard-Bergery Einstein metrics on these spaces.  Maschler, in \cite{Mas}, suggested that this was likely to be the case, and it is probably known to experts. What is more surprising is that we are able to show that the K\"ahler metrics always lie in a multiple of the first Chern class. In this way the L\"u-Page-Pope metrics  are similar to the Dancer-Wang Ricci solitons, since  any K\"ahler-Ricci soliton must lie in the first Chern class.\\
\\
%\newpage 
The next, and most significant, part of the article makes the link with K\"ahler geometry even more explicit for the  lowest dimensional case of the L\"u-Page-Pope construction. Here the underlying manifold in this case is the non-trivial $\mathbb{C}P^1$-bundle over $\mathbb{C}P^1$, which can also be described as the one-point blow up of the complex projective plane, $\mathbb{C}P^2\sharp \overline{\mathbb{C}P}^{2}$. The Einstein metric given by this construction was originally discovered by Page \cite{Page}, and the associated conformally K\"ahler metric is due to Calabi \cite{Cal}.  The K\"ahler-Ricci soliton on this manifold was originally discovered independently by Koiso \cite{Koi} and Cao \cite{Cao}. All of these K\"ahler metrics are toric, and therefore have a beautiful description due to Abreu \cite{Abr1}, \cite{Abr2} and Guillemin \cite{Gui}. In section \ref{Sect3}, we show that the L\"u-Page-Pope metrics can also be explicitly described in this framework. This has a number of consequences; it leads to a greatly simplified proof of the existence of the quasi-Einstein metrics, and also gives a straightforward proof the results of section \ref{Sect2} in this special case. Moreover, our construction provides a unified framework for constructing the Page, Koiso-Cao, and L\"u-Page-Pope metrics in one fell swoop.\\
\\
In section \ref{Sect4}, the  related toric surface $\mathbb{C}P^2\sharp2 \overline{\mathbb{C}P}^{2}$ is studied. This manifold is known to admit a conformally K\"ahler Einstein metric, analogous to the Page metric, due to Chen-LeBrun-Weber \cite{CLW}. It also admits a K\"ahler-Ricci soliton, analogous to the Koiso-Cao metric, due to Wang and Zhu \cite{WZ} (Donaldson gives an alternative proof of the existence theorem using the Abreu-Guillemin framework in \cite{Don}). The problem of constructing quasi-Einstein metrics analogous to the L\"u-Page-Pope metrics on $\mathbb{C}P^2\sharp2 \overline{\mathbb{C}P}^{2}$ is a natural open problem. If a family of quasi-Einstein metric with the same properties as the L\"u-Page-Pope were to exist on $\mathbb{C}P^2\sharp \overline{\mathbb{C}P^2}$, then our methods allow us to determine the explicit form of the potential function $\phi$ for metrics in a given cohomology class.\\
\\
Finally in section \ref{Sect5} we discuss some open problems and areas for future research.\\
\\
\textit{Acknowledgements}:
The majority of this work was undertaken whilst WB and TM  paid research visits to SH in December 2014 and January 2015.  The visit of WB was supported by a Scheme 5 grant from the London Mathematical Society. The visit of TM was funded by a Dennison research grant from the University of Buckingham and a grant from California State University Fullerton. The authors wish to warmly thank Gideon Maschler for useful comments on an early draft of this article.

\section{The geometry of L\"u-Page-Pope metrics}\label{Sect2}
In this section we make precise the relationship between the L\"u-Page-Pope quasi-Einstein metrics on certain $\mathbb{C}P^{1}$-bundles and the K\"ahler geometry of such manifolds.
\subsection{Construction of metrics on the manifolds $W_{q}$}
We begin by describing the construction of the manifolds $W_{q}$.  Let $(M,h,J)$ be a Fano K\"ahler-Einstein manifold of complex dimension $n$. Write the first Chern class of $M$ as $c_{1}(M) = pa$, where $p\in \mathbb{N}$ and $a\in H^{2}(M,\mathbb{Z})$ is an indivisible class. For example, in the case $M=\mathbb{C}P^{n}$, one has  $p=n+1$ and $a = c_{1}(\mathcal{O}(1)$). The metric $h$ is normalised so that $$Ric(h) = ph.$$ In other words, if $\eta(\cdot,\cdot)=h(J\cdot,\cdot)$ is the K\"ahler form of $h$, then $[\eta]=a$.\\
\\
Denote by $P_{q}$ the principal $U(1)$-bundle over $M$ with Euler class $e = qa$ where $q\in \mathbb{Z}$. Let $\theta$ be the connection with curvature ${\Omega = q\eta}$. Finally, denote by $W_{q}$ the projectivization $\mathbb{P}(L_{q}\oplus \mathcal{O})$ where $L_{q}$ is the associated holomorphic line-bundle of $P_{q}$. It is useful to view the manifolds $W_{q}$ as the compactification of ${P_{q}\times (0,4)}$ obtained by collapsing a $U(1)$-fiber at $0$ and $4$.  This gives rise to Riemannian metrics on $W_{q}$ of the form
\begin{equation}\label{metrics}
g = \alpha(s)^{-1}ds^{2}+\alpha(s)\theta \otimes\theta + \beta(s)\pi^{\ast}h,
\end{equation}
where $s$ is the coordinate on $(0,4)$, $\pi: {W_{q}\rightarrow M}$ is the projection and ${\alpha, \beta \in C^{\infty}((0,4))}$.  In order for the metrics of the form (\ref{metrics}) to extend smoothly to the compactification $W_{q}$, the functions $\alpha$ and $\beta$ satisfy $$\alpha(0)=\alpha(4)=0 \textrm{ and } \alpha'(0)=-\alpha'(0)=2.$$

The precise theorem that guarantees existence of non-trivial quasi-Einstein metrics on the manifolds $W_{q}$ is Theorem 3 in \cite{Hallqem} (the case when $m$ is integral  was first proved in \cite{LPP}). In the case where the base manifold is a single factor this can be restated as:
\begin{theorem}
For $0<|q|<p$, let $W_{q}$ be as described above. Then, for all $m>1$, there exists a non-trivial quasi-Einstein metric of the form (\ref{metrics}) on $W_{q}$. Furthermore the function $\beta$ is given by
$$\beta(s) = A(s+s_{0})^{2}-\frac{q^{2}}{4A},$$
where $s_{0}$ and $A$ are constants satisfying
\begin{equation}\label{Aeq}
s_{0}(s_{0}+4) = \frac{8Ap+q^{2}}{4A^{2}}.
\end{equation}
\end{theorem}
\begin{remark}
Theorem 3 in \cite{Hallqem} involves a constraint, the non-vanishing of a certain integral, which the L\"u-Page-Pope examples automatically satisfy.  The constraint is suggestive of a link between the quasi-Einstein metrics on $W_{q}$ and K\"ahler geometry as the integral is essentially the Futaki invariant.   A nontrivial K\"ahler-Ricci soliton must have non-vanishing Futaki invariant.
\end{remark} 
\subsection{The complex geometry of $W_{q}$}
The complex structure on $W_{q}$ can be described in the $(s, \theta)$-coordinates. One can lift the complex structure $J$ on the base component and then define $J(\partial_{s}) = -(1/\sqrt{\alpha})\partial_{\theta}$. Hence the K\"ahler form of Equation (\ref{metrics}) is given by
\begin{equation}
\omega = \theta \wedge ds +\beta\pi^{\ast}\eta.
\end{equation}
\begin{lemma}[cf. Corollary 7.3 in \cite{WW}]\label{confKlem}
Let $g$ be a L\"u-Page-Pope metric. If
$$\sigma(s) = -\log\left(|2A(s+s_{0})-q|\right),$$
then the  conformally related metric $g_{K}=e^{2\sigma}g$ is K\"ahler.
\end{lemma}
\begin{proof}
This follows from straightforward calculation. The Hermitian form of $g_{K}$ is given by
$$\omega_{K} = e^{2\sigma}(\theta \wedge ds +\beta\pi^{\ast}\eta)$$
and so
$$d\omega_{K} = e^{2\sigma}(d\theta \wedge ds + (\beta'(s)+2\sigma'(s)\beta(s))ds \wedge \pi^{\ast}\eta).$$
Using the fact that $d\theta = \eta$, this vanishes if
$$\beta'(s)+2\sigma'(s)\beta(s)+q = 0.$$
Hence as $\beta(s) = A(s+s_{0})^{2}-q^{2}/4A$  it follows that (up to a constant)
$$\sigma(s) = -\log\left(|2A(s+s_{0})-q|\right).$$
\end{proof}
In order to compute the first Chern class of $W_{q}$ we revert to considering the manifold as the projectivisation of a rank two holomorphic vector bundle, namely $\mathbb{P}(L_{q}\oplus \mathcal{O})$. Some of the topology we need is presented in section 6 of \cite{WW}. Over each $\mathbb{C}P^{1}$-fiber there is the tautological line bundle $\mathcal{O}_{\mathbb{C}P^{1}}(-1)$. We denote the first Chern class of the dual of this line bundle over $W_{q}$ by $F = c_{1}(\mathcal{O}_{\mathbb{C}P^{1}}(1))$.  The Leray-Hirsch theorem states that 
$H^{2}(W_{q};\mathbb{Z}) \cong H^{2}(M;\mathbb{Z})\oplus \langle F\rangle$. In such a setting, we have the following lemma.
\begin{lemma}[Proposition 6.4 in \cite{WW}]\label{Cclasslem}
Let $W_{q} = \mathbb{P}(L_{q}\oplus \mathcal{O})$ and $F$ be as described above. Then
$$c_{1}(W_{q}) = (p+q)\pi^{\ast}a+2F.$$
\end{lemma}
We can now compute the cohomology class of the K\"ahler metric $g_{K}$ conformal to the L\"u-Page-Pope metric.
\begin{theorem}\label{mainT}
Let $(W_{q}, g) $ be a L\"u-Page-Pope metric.  Then
\begin{enumerate}
\item the metric $g$ is conformal to a K\"ahler metric $g_{K}$, and
\item the cohomology class of the associated class $\omega_{K}$ is a scalar multiple of the first Chern class of $W_q$.
\end{enumerate}
\end{theorem}

\begin{proof}
The first part of the Theorem follows immediately from Lemma \ref{confKlem}. In order to compute the cohomology class of the metric some distinguished homology classes will be introduced.  Let $\tau_{1}\in H_{2}(M,\mathbb{R})$ be the class dual to $a \in H^{2}(M,\mathbb{R})$, in the sense that $\int_{\tau_{1}}a = 1$.  
Similarly, let $\tau_{2}\in H_{2}(W_{q},\mathbb{R})$ be the homology class dual to $F$, which means  that $\int_{\tau_{2}}F = 1$ (Here one represents $\tau_2$ by a $\mathbb{C}P^1$ fibre divided by $2\pi$.).   Denote by $\tau_{1}^{0}$ the class of $\tau_{1}$ in the copy of $M$ glued in to $W_q$ at $s=0$ and by $\tau_{1}^{4}$ the $\tau_{1}$ in the copy of $M$ glued in at $s=4$. We claim that these classes satisfy the equation 
\begin{equation}\label{eulerclass}
\tau_{1}^{0}-\tau_{1}^{4}=q\tau_{2}.
\end{equation}

Establishing the claim  starts with the fact that $H_{2}(W_q,\mathbb{Z})  \cong H_2(M, \mathbb{Z} )\oplus \tau_{2}$. Here  the first factor represent the pushforward of classes in the base M via a generic section of the vector bundle $L_{q}$, and $\tau_{2}$ is the homology class of the $\mathbb{C}P^1$ fibre.

The copy of $M$ at $s=0$  represents  the zero section of the line bundle $L_{q}$. The class $\tau_{1}^{4}$ obviously does not intersect $\tau_{1}^{0}$ in homology as the manifolds $M$ at $s=0$ and $s=4$ do not intersect. If we restrict to the bundle defined over a representative cycle of the homology class $\tau_{1}^{0}$ then, because the pullback of the first Chern class is the first Chern class of the pullback,  we see that  generic sections of the restriction of $L_{q}$ intersect $q$ times. Hence the homology of the subbundle is  generated by $\tau^{0}_{1}$ and $\tau_{2}$ with

$$
\tau_{1}^{0} \cap \tau_{1}^{0} = q \text{ , }  \tau_{2}\cap \tau_{1}^{0} = 1 \text{ ,  and } \tau_{2}\cap \tau_{2}=0
$$
where $\cap$ denotes oriented intersection.  Hence if $\tau_{1}^{0}\cap (a\tau_{1}^{0}+b\tau_{2}) = 0$, it follows that  $a = 1$ and $b=-q$ because the coefficients are elements of $\mathbb{Z}$.  Hence $\tau^{4}_{1} = \tau_{0} -q\tau_{2}$.  By construction, none of the homology classes $\tau^1_{0}$,$\tau^1_{4}$ or $\tau_{2}$ vanish when they are embedded in $W_{q}$, as they were defined as the classes one gets in the image of this embedding. Therefore  this identity must also hold in $H_{2}(W_{q})$ and  Equation (\ref{eulerclass}) is established.

Using  Equation \ref{eulerclass},  Lemma \ref{Cclasslem} can be restated as
$$\int_{\tau_{1}^{0}}c_{1}(W_{q})  = p+q \textrm{ and } \int_{\tau_{1}^{4}}c_{1}(W_{q}) = p-q.$$
This implies that in order to prove the theorem one needs to evaluate the metric on the copies of $\tau$ at $s=0$ and at $s=4$ and take the ratio.  We compute
$$\frac{e^{2\sigma(4)}\beta(4)}{e^{2\sigma(0)}\beta(0)} = \frac{(2As_{0}-q)^{2}(4A^{2}(4+s_{0})^{2}-q^{2})}{(2A(4+s_{0})-q)^{2}(4A^{2}s_{0}^{2}-q^{2})} = \frac{(2As_{0}-q)(2A(4+s_{0})+q)}{(2A(4+s_{0})-q)(2As_{0}+q)}.$$
Using (\ref{Aeq}) we see
$$\frac{e^{2\sigma(4)}\beta(4)}{e^{2\sigma(0)}\beta(0)} = \frac{8Ap+q^{2}-8Aq-q^{2}}{8Ap+q^{2}+8Aq-q^{2}} = \frac{p-q}{p+q}.$$
The result now follows.

\end{proof}
\begin{remark}  In \cite{Mas}, conformally K\"ahler quasi-Einstein metrics were investigated by Maschler. He showed that in complex dimension 3 and greater, assuming the K\"ahler metric is not a local product, the square root of the conformal factor is a Killing potential and the potential function and the conformal factor are fucntionally dependent, then the manifold is biholomorphic to one of the manifolds $W_{q}$. 
\end{remark}

\section{Metrics on  $\mathbb{C}P^{1}\rightarrow\mathbb{C}P^{1}$ in the Abreu-Guillemin framework}\label{Sect3}
In this section we study in more detail the lowest-dimensional example of the L\"u-Page-Pope construction which occurs on the non-trivial $\mathbb{C}P^{1}$-bundle over $\mathbb{C}P^{1}$. As mentioned in the introduction, we shall switch perspectives and consider this manifold as the blow-up of the complex projective plane $\mathbb{C}P^{2}$ at one point, written $\mathbb{C}P^{2}\sharp\overline{\mathbb{C}P}^{2}$ . In this section we write the conformally K\"ahler L\"u-Page-Pope metric (and Page's conformally Einstein metric and the Koiso-Cao soliton) explicitly in symplectic (also known as action-angle) coordinates. This has a number of nice features; it simplifies the existence theory in \cite{LPP} and \cite{Hallqem} and the results of Theorem \ref{mainT} are almost immediate in this setting. It also suggests how the existence theory might run on the other toric Fano surface with non-vanishing Futaki invariant, $\mathbb{C}P^{2}\sharp2\overline{\mathbb{C}P}^{2}$.
\subsection{$U(2)$-invariant K\"ahler Metrics on $\mathbb{C}P^{2}\sharp\overline{\mathbb{C}P}^{2}$}
To begin with, we consider K\"ahler metrics that are invariant under a Hamiltonian action by the torus $\mathbb{T}^{2}$. The moment polytope (i.e. the image of the moment map) for the manifold $\mathbb{C}P^{2}\sharp\overline{\mathbb{C}P}^{2}$ is the trapezium (trapezoid) $T$ described by the linear inequalities
$$l_{1}(x) = (1+x_{1}), \ \ l_{2}(x) = (1+x_{2}), \ \ l_{3}(x) = (1-x_{1}-x_{2}), \ \ l_{4}(x) = (a+x_{1}+x_{2}).$$
The parameter $a\in (-1,2)$ determines the volume of the exceptional divisor and hence the cohomology class that the associated metric $\omega$ is in.  The case where $a=1$ corresponds to the case when $[\omega] = c_{1}$. We note in this case, the volume of the exceptional divisor is one third that of the volume of the projective line at infinity. The Guillemin theory states that there is an open set in the manifold, diffeomorphic to $T^{\circ}\times \mathbb{T}^{2}$ where the metric takes the form
\begin{equation}\label{torKmet}
g = u_{ij}dx_{i}dx_{j}+u^{ij}d\theta_{i}d\theta_{j},
\end{equation}
where $u$ is a function on $T$ known as the \textit{symplectic potential}. Here $u_{ij}$ is the Hessian matrix in the Euclidean coordinates $x_{1},x_{2}$ on the trapezium $T$, and $u^{ij}$ is the inverse matrix. Furthermore, Guillemin showed that the symplectic potential $u$ has the form
\begin{equation}\label{sympot}
u(x) = \frac{1}{2}\left(\sum_{i=1}^{i=4}l_{i}(x)\log(l_{i}(x))+f(x_{1},x_{2})\right),
\end{equation}
where $f$ is a smooth function with all derivatives continuous up to the boundary $\partial T$. The manifold $\mathbb{C}P^{2}\sharp\overline{\mathbb{C}P}^{2}$ inherits a $U(2)$ action from $\mathbb{C}P^{2}$ which fixes the point that is blown up. Hence the action lifts to $\mathbb{C}P^{2}\sharp\overline{\mathbb{C}P}^{2}$.  It is an example of a cohomogeneity one action as the orbit of a generic point is a three-sphere $\mathbb{S}^{3}$. If we restrict to $U(2)$-invariant K\"ahler metrics then one can take $f(x_{1},x_{2}) = f(x_{1}+x_{2})$ in (\ref{sympot}). For the remainder of the article we will take $t=x_{1}+x_{2}$. We can write the metric explicitly by noting that the Euclidean Hessian of $u$ is given by
$$D^{2}u = \frac{1}{2}\left[\begin{array}{cc} \frac{1}{x_{1}+1}+P(t) & P(t)\\ P(t) & \frac{1}{x_{2}+1}+P(t) \end{array}\right],$$
where
$$P(t) = \frac{1}{1-t}+\frac{1}{a+t}+f''(t).$$
It will also be useful to introduce, in terms of $t=x_{1}+x_{2}$ the related functions
$$F(t) = 1+(2+t)P(t) \textrm{ and } z(t) = F^{-1}(t).$$
The function $z(t)$ satisfies the following conditions at the boundaries
\begin{equation}\label{Bound}
z(-a)=z(1) = 0 \textrm{ and } z'(-a)=(2-a)^{-1}, z'(1)=-1/3.
\end{equation}
The determinant of the metric and the inverse of the matrix $D^{2}u$ are given by
\begin{equation}\label{Det}
\det(D^{2}u) = \frac{F(t)}{4(x_{1}+1)(x_{2}+1)},
\end{equation}
and
\begin{equation}\label{inverse}
(D^{2}u)^{-1} = \frac{2(x_{1}+1)(x_{2}+1)}{F(t)}\left[ \begin{array}{cc} \frac{1}{x_{2}+1}+P(t) & -P(t) \\ -P(t) & \frac{1}{x_{1}+1}+P(t) \end{array}\right].
\end{equation}
The $x_{i}$-components of the Ricci tensor in these coordinates are given by
\begin{equation}\label{Ricci}
\Ric_{ij}=\frac{1}{2}\left(\frac{\partial^{2}}{\partial x_{i}\partial x_{j}}-u^{kl}\frac{\partial u_{ij}}{\partial x_{k}}\frac{\partial}{\partial x_{l}}\right)\log \det(D^{2}u).
\end{equation}
The following quantity will be especially useful in our calculations
\begin{equation}\label{RicODE}
\Ric_{11}-\Ric_{22} = \frac{1}{2}\frac{x_{2}-x_{1}}{(x_{1}+1)(x_{2}+1)}\left(\frac{F'}{F^{2}}+\frac{2(F-1)}{F(2+t)} \right).
\end{equation}
Functions on the manifold that are invariant under the $U(2)$ action can be expressed as functions ${\phi(t):[-a,1]\rightarrow \mathbb{R}}$. We will also need a similar expression to the one above for the Hessian (calculated with the metric (\ref{torKmet}) of such functions:
\begin{equation}\label{HesODE}
\nabla^{2}\phi_{11}-\nabla^{2}\phi_{22} = \frac{1}{2}\frac{x_{2}-x_{1}}{(x_{1}+1)(x_{2}+1)}\frac{\phi'}{F}.
\end{equation}
Throughout this paper we will use the analyst's Laplacian $\Delta = tr(\nabla^{2})$. The Laplacian of a $U(2)$-invariant function $\phi(t)$ is given by
\begin{equation}\label{DelODE}
\Delta \phi = \left(-2(2+t)\frac{F'}{F^{2}}+\frac{4}{F}\right)\phi' + \frac{2(2+t)}{F}\phi''.
\end{equation}
We will also need the formulae for how some of the above quantities transform under a conformal rescalling of the metric. If $\tilde{g} = e^{2\sigma}g$ for $\sigma \in C^{\infty}(M)$ then
\begin{equation}\label{confRic}
\Ric (\tilde{g}) = \Ric (g) -2\left(\nabla^{2}\sigma-d\sigma\otimes d\sigma\right)-(2|\nabla \sigma |^{2}+\Delta\sigma)g,
\end{equation}
where all the quantities on the righthand side are computed with the metric $g$. The Hessian of a function $\phi$ tranforms under conformal rescaling via
\begin{equation}\label{confHes}
\widetilde{\nabla}^{2}\phi = \nabla^{2}\phi +e^{\sigma}[de^{-\sigma} \otimes d\phi + d\phi\otimes de^{-\sigma}-g(\nabla e^{-\sigma},\nabla\phi)g],
\end{equation}
and hence the Laplacian transforms via
\begin{equation}\label{confLap}
\widetilde{\Delta}\phi = e^{-2\sigma}\left(\Delta\phi+2g(\nabla\sigma,\nabla\phi)\right).
\end{equation}

\subsection{Explicit metrics}
We now determine explicit representations for the function $z(t)$.  One could then rearrange and perform the required integration in order to determine the function $f(t)$ in the symplectic potential (\ref{sympot}).
\subsubsection{The L\"u-Page-Pope metrics}
As before, we write $g_{LPP} = e^{2\sigma}g_{K}$, with $g_{K}$ a K\"ahler metric. The L\"u-Page-Pope metrics have $J$-invariant Ricci tensor and so the function $e^{-\sigma}$ is a Killing potential and thus $\sigma(t) = -\log(bt+c)$ for constants $b$ and $c$. We can take the potential function to be invariant under the $U(2)$-action and so $\phi$ is also a function of $t$. We are hence in the setting considered by Maschler  and so by the discussion following Equation (2.2) in \cite{Mas} we have 
$$\phi(t) = -m\log(e^{\sigma(t)}+d) = -m\log\left( \frac{dbt+dc+1}{bt+c}\right),$$
for a constant $d$. We note that the constants $a,b,c,d$ must satisfy
\begin{equation}\label{Positive}
bt+c>0 \textrm{ and } dbt+dc+1>0,
\end{equation}
for all $t\in [-a,1]$. Using Equations (\ref{confRic}), (\ref{confHes}) and (\ref{confLap}), for conformally related quantities,  together with Equations (\ref{RicODE}), (\ref{HesODE}) and (\ref{DelODE}), we can rewrite the equation
$$ Ric_{11}(g_{LPP})-Ric_{22}(g_{LPP})+\nabla^{2}\phi_{11}-\nabla^{2}\phi_{22}=(g_{LPP})_{11}- (g_{LPP})_{22}$$ as a first-order ODE
for $z(t)$:
\begin{equation}\label{LPPZODE}
\begin{split}
\frac{dz}{dt} + & \left(\frac{2}{2+t}-\frac{3b}{bt+c}-\frac{b}{bt+4b-c}-\frac{mb}{(dbt+dc+1)(bt+c)}\right)z\\
&+\frac{2(bt+c)^{2}-(2+t)}{(2+t)(bt+c)(bt+4b-c)} = 0.
\end{split}
\end{equation}

Quasi-Einstein metrics have a first integral due to Kim and Kim \cite{KK} coming from the contracted second Bianchi indentity. For any solution to (\ref{QEMeq}), there is a constant $\mu$ for which the quasi-Einstein potential $\phi$ satisfies
\begin{equation}\label{KK}
1-\frac{1}{m}\left(\Delta\phi-\|\nabla \phi\|^{2}\right)=\mu e^{\frac{2\phi}{m}}.
\end{equation}
We calculate with respect to the K\"ahler metric and obtain another ODE
\begin{equation}
\begin{split}
\frac{dz}{dt} + & \left(\frac{2}{2+t}-\frac{3b}{bt+c} -\frac{b}{bt+c+d^{-1}} -\frac{mb}{(dbt+dc+1)(bt+c)} \right)z \\
& +\left(\frac{\mu(bt+c)^{2}-(dbt+dc+1)^{2}}{2bd(2+t)(bt+c)(bt+c+d^{-1})}\right)= 0.
\end{split}
\end{equation}
In order to have consistency we must have
$$d=(2(2b-c))^{-1} \textrm{ and } \mu = d^{2}+4bd.$$
Using the boundary conditions (\ref{Bound})  we obtain
$$\frac{3-2(b+c)^{2}}{3(b+c)(5b-c)} = -\frac{1}{3} \textrm{ and } \frac{(2-a)-2(c-ab)^{2}}{(2-a)(c-ab)((4-a)b-c)}=\frac{1}{(2-a)}.$$
Rearranging we see that
$$c^{2}=b^{2}+1 \textrm{ and }  c^{2}=a(4-3a)b^{2}+4(a-1)bc+(2-a).$$
If $a\neq 1$ then $c=(2\pm\sqrt{3(2-a)})b$.  In the case $c =(2+\sqrt{3(2-a)})b$, then $c>2b$ and so $d=(2(2b-c))^{-1} <0$. This means that
$$bt+c+d^{-1} = bt+4b-c<0$$ 
on the interval $[-a,1]$. Hence $c>5b$; this is a contradiction as $a\in (-1,2)$. In the case where $c=(2-\sqrt{3(2-a)})b$ we have $b+c>0$ and so $(3-\sqrt{3(2-a)})b>0$. As $-1<a<2$ this means we must have $b>0$.  We also have $c-ba>0$ hence $2-\sqrt{3(2-a)} >a$ and so $a<-1$. This is a contradiction.  Hence $a=1$ and we have proved in a straightforward fashion the second part of Theorem \ref{mainT}.\\
\\
The solution of (\ref{LPPZODE}) is given by
\begin{equation*}
z(t) =\frac{d(bt+c)^{m+3}}{(dbt+dc+1)^{m-1}(2+t)^{2}}  \int_{-1}^{t}\frac{(2+s)-2(bs+c)^{2}}{(bs+c)^{m+4}}(dbs+dc+1)^{m-2}(2+s)ds,
\end{equation*}
where $c = \sqrt{1+b^{2}}$ and $d = (2(2b-c))^{-1}$.
In order that we get a smooth metric we must be able to choose a compatible $b$ so that $z(1)=0$, which is equivalent to
$$I(b) :=\int_{-1}^{1}\frac{(2+s)-2(bs+c)^{2}}{(bs+c)^{m+4}}(dbs+dc+1)^{m-2}(2+s)ds=0.$$
We note that
$$I(0) = \left(\frac{1}{2}\right)^{m-2}\left( \frac{2}{3}\right) >0$$
and, for $m>1$,
$$I(\frac{1}{\sqrt{24}}) = -\frac{1}{12}\int_{-1}^{1}\frac{(s-1)^{2}}{(bs+c)^{m+4}}(dbs+dc+1)^{m-2}(2+s)ds<0.$$
Hence we see that there exists $b\in (0,1/\sqrt{24})$ giving a smooth solution of (\ref{LPPZODE}), and hence of (\ref{QEMeq}).
For a fixed value of $m$ it is very easy to find the approximate values of $b$ (and hence $c$ and $d$) numerically.  When $m=2$ we find
$$b\approx 0.076527, \ \ \ c\approx 1.002924 \textrm{ and } d \approx -0.588325.$$
For $m=50$ we find
$$b\approx 0.005120, \ \ \ c\approx 1.000013 \textrm{ and } d \approx -0.505167.$$
Using the above values for $m=50$ we compute 
$$\frac{\phi(1)-\phi(-1)}{2}\approx 0.517374,$$
which suggests $\phi$ is approximately a linear function with gradient close to $0.52$. If one compares the construction of the Koiso-Cao soliton $(g_{KC}, \phi_{KC})$ in subsection \ref{KCsol} we see there is strong numerical evidence that as $m\rightarrow \infty$
$$\sigma \rightarrow 0, \ \ \ \phi \rightarrow \phi_{KC} \textrm{ and } g_{LPP} \rightarrow g_{KC},$$
where one expects to get convergence in the $C^{\infty}$ topology. A careful study of the above ansatz and the dependence of the values of $b$ upon $m$ would probably yield this result.      
\subsubsection{Page's Einstein metric}
Toric constructions of the K\"ahler metric $g_{K}$ conformal to Page's Einstein metric $g_{P}$ were already given by Abreu \cite{Abr1} and Dammerman \cite{Dam}.  Both authors constructed Calabi's family of extremal metrics on the manifold which involves solving a fourth-order ODE. Our approach via the function $z(t)$ is slightly different as it uses only the Einstein equation and is therefore second-order. As with the L\"u-Page-Pope metrics we write $g_{P}=e^{2\sigma}g_{K}$.  As the Page metric must have $J$-invariant Ricci tensor where $J$ is the complex structure, the gradient $\nabla_{K}e^{-\sigma}$ must be a holomorphic vector field.  As the function $\sigma$ is also $U(2)$-invariant we find
$$\sigma(t) = - \log(bt+c),$$
where $b,c $ are constants that satisfy $bt+c>0$ for $t\in [-a,1]$. We can always perform a homothetic rescaling to fix $Ric(g_{P})=g_{P}$ and, as in the L\"u-Page -Pope case, rewrite the equation $$Ric_{11}(g_{P})-Ric_{22}(g_{P})=g_{P11}-g_{P22}$$ as a first-order ODE  for the function $z(t)$:
\begin{equation}
\frac{dz}{dt}+\left(\frac{2}{2+t}-\frac{3b}{bt+c}-\frac{b}{bt+4b-c}\right)z+\left( \frac{2(bt+c)^{2}-(2+t)}{(2+t)(bt+c)(bt+4b-c)}\right)=0.
\end{equation}
This has a solution  $z(t)=\frac{\mathcal{A}(t)}{\mathcal{B}(t)}$, with 
\begin{equation*}
\begin{split} 
\mathcal{A}(t)  = & 6 Kb^{7}t^{4}+24 Kb^{7} t^{3}+12K b^{6}ct^{3}+72Kcb^{6}t^{2}+72Kb^{5} c^{2}t \\
&-12Kb^{4}c^{3}t+24Kb^{4}c^{3}-6Kb^{3}c^{4}+6b^{3} t^{2}+12b^{2}ct+6bc^{2}-2bt-2b-c,\\
\mathcal{B}(t) =  &(6b^{3} (t+2)^{2}),
\end{split}
\end{equation*} 
and $K$ a constant.
In order to determine the constants we consider the boundary behaviour of $z(t)$.  
Using the form of the solution and the boundary conditions we let 
$$z(t) = \frac{(t-1)(t+a)(At^{2}+Bt+C)}{(t+2)^{2}}.$$
Then using the conditions on the derivative
$$A+B+C = \frac{-3}{(1+a)} \textrm{ and } a^{2}A-aB+C = \frac{-(2-a)}{a+1}.$$
Using the form of $z(t)$ we obtain
$$A = Kb^{4}, B+(a-1)A =  4Kb^{4}+2Kb^{3}c \textrm{ and } -aA+(a-1)B+C = 12Kb^{3}c+1 .$$
This yields
$$(30-7a)A+(a-7)B+C=1.$$
Hence
$$
\left(\begin{tabular}{ccc}
1 & 1 & 1\\
$(30-7a)$ & $(a-7)$ & 1\\
$a^{2}$ & $-a$ & 1
\end{tabular}\right) 
\left(\begin{tabular}{c}
A\\
B\\
C
\end{tabular}\right)=
\left(\begin{tabular}{c}
$\frac{-3}{1+a}$\\
1\\
$\frac{-(2-a)}{(1+a)}$
\end{tabular}\right).
$$
Solving this system yields
$$A =\frac{2(a-2)}{(1+a)(a^{2}-16a+37)} ,\ \ B =\frac{a^{2}+10a-33}{(1+a)(a^{2}-16a+37)}, $$
and $$ C=\frac{-2(2a^{2}-18a+37)}{(1+a)(a^{2}-16a+37)}.$$
From this we can deduce
$$\frac{c}{b} = \frac{3a^{2}-4a-13}{4(a-2)}.$$
On the other hand, the conditions (\ref{Bound}) yield
$$\frac{3-2(b+c)^{2}}{3(b+c)(5b-c)} = -\frac{1}{3} \textrm{ and } \frac{(2-a)-2(c-ab)^{2}}{(2-a)(c-ab)((4-a)b-c)}=\frac{1}{(2-a)}.$$
Rearranging we see that
$$c^{2}=b^{2}+1 \textrm{ and }  c^{2}=a(4-3a)b^{2}+4(a-1)bc+(2-a).$$
Hence
$$(a-1)((1-3a)b^{2}+4bc-1)=0.$$
If $a=1$ then $c=7b/2$ and 
$$A=-\frac{1}{22}, \ \ \  B=-\frac{1}{2}, \textrm{ and }   C = -\frac{21}{22}.$$
By examing the form of the solution $z(t) = \mathcal{A}(t)/\mathcal{B}(t)$ one can see that this choice leads to inconsistency. 
If $a\neq 1$ we have ${c=(2\pm\sqrt{3(2-a)})b}$ and we deduce
$$3a^{4}-8a^{3}-42a^{2}+168a-125=0.$$
and so $a \approx 1.057769$. We note with this value of $a$, that the ratio of the volume of the exceptional divisor and the projective line that does not intersect it is 
$$\frac{3}{2-a}\approx 3.183933$$
which agrees with values calculated by other methods in \cite{Bes} and \cite{LBa}.
\subsubsection{The Koiso-Cao K\"ahler Ricci soliton}\label{KCsol}
We now look for a $U(2)$-invariant solution to the equation
\begin{equation}\label{Soleqn}
\Ric (g) +\nabla^{2}\phi=\lambda g.
\end{equation}
If we assume the metric is K\"ahler this forces the function $\phi$ to have holomorphic gradient and so we can assume ${\phi=c(x_{1}+x_{2})}$ for some constant $c$.  We will also factor out homothety by setting $\lambda =1$. 
Hence we obtain the following equation for $z(t)$:
\begin{equation}\label{KCBernoulli2}
\frac{dz}{dt}+\frac{2-c(2+t)}{(2+t)}z+\frac{t}{2+t}=0.
\end{equation}
This yields 
$$z(t) = \frac{de^{c(t+2)}}{(t+2)^{2}}+\frac{c^{2}t(t+2)+2c(t+1)+2}{c^{3}(t+2)^{2}},$$
where $d$ is a constant.
Using the boundary conditions $z(1)=z(-1)=0$, we get the equations
$$de^{c} + \frac{2-c^{2}}{c^{3}}=0$$
and
$$\frac{de^{3c}}{9}+\frac{3c^{2}+4c+2}{9c^{3}}=0.$$
This means that $c$ solves
$$e^{2c}(c^{2}-2)+3c^{2}+4c+2=0,$$
which yields $c \approx 0.5276$ and $d\approx -6.91561$. This agrees with the value found by other methods in \cite{HallDGA}. 
\\
As with the L\"u-Page-Pope metric, one could also recover the relevant equations by considering a 1st integral due to Hamilton \cite{Ham} and Ivey \cite{Ivey}; namely
$$\Delta_{\phi}\phi+2\phi=\Delta \phi-|\nabla \phi|^{2}+2\phi= 0,$$
where $\phi$ is normalised so that $\int_{M}\phi e^{-\phi} = 0.$ 
This yields
$$\left(-2(2+t)\frac{F'}{F^{2}}+\frac{4}{F}\right)c-\frac{2(2+t)}{F}c^{2}+2ct=0.$$
Hence
$$\frac{dz}{dt}+\left(\frac{2}{2+t}-c \right)z+\frac{t}{2+t}=0.$$
\section{Metrics on $\mathbb{C}P^{2}\sharp 2\overline{\mathbb{C}P}^{2}$}\label{Sect4}
We now perform a very similar analysis on the toric surface $\mathbb{C}P^{2}\sharp 2\overline{\mathbb{C}P}^{2}$. In this case the moment polytope is the pentagon $P$ given by $l_{i}(x)>0$ for the linear functions:
\begin{equation*}
\begin{split}
l_{1}(x) &= 1+x_{1}\text{ , } l_{2}(x) = 1+x_{2}\text{ , }  l_{3}(x) = a-1-x_{1},\\ 
l_{4}(x) &= a-1-x_{2},  l_{5}(x) = a-1-x_{1}-x_{2}.
 \end{split}
 \end{equation*}
 
Here we assume that the metric is also symmetric under the $\mathbb{Z}_{2}$ action that swaps $x_{1}$ and $x_{2}$.  This is sensible as both the Wang-Zhu K\"ahler-Ricci soliton and the Chen-LeBrun-Weber metric have this symmetry. A $\mathbb{Z}_{2}$-invariant toric K\"ahler metric $g$ on this manifold can be written in symplectic coordinates  given by Equation (\ref{torKmet}) with
\begin{equation*}\label{torKcp2b2}
D^{2}u = \left(\begin{array}{cc} \frac{(a^{2}-ax_{2}-x_{1}^{2}-2x_{1}-1)}{2(x_{1}+1)(a-1-x_{1})(a-1-x_{1}-x_{2})}+f_{11} & \frac{1}{2(a-1-x_{1}-x_{2})}+f_{12}\\ \frac{1}{2(a-1-x_{1}-x_{2})}+f_{12} & \frac{(a^{2}-ax_{1}-x_{2}^{2}-2x_{2}-1)}{2(x_{2}+1)(a-1-x_{2})(a-1-x_{1}-x_{2})}+f_{22}\end{array}\right),
\end{equation*}
where $f:P\rightarrow\mathbb{R}$ is a smooth function with $f(x_{1},x_{2})=f(x_{2},x_{1})$. One can show that the determinant of $g$ is given by
$$\det (g) = \frac{\mathcal{D}}{4(x_{1}+1)(x_{2}+1)(a-1-x_{1})(a-1-x_{2})(a-1-x_{1}-x_{2})}$$
where 
\begin{equation*}
\begin{split}
\mathcal{D}= &a(a^{2}+a-(x_{1}^{2}+x_{2}^{2})-2(x_{1}+x_{2})-2)\\
&+2(x_1+1)(a-1-x_{1})P_{2}f_{11}+2(x_{2}+1)(a-1-x_{2})P_{1}f_{22}\\
& -4(x_{1}+1)(x_{2}+1)(a-1-x_{1})(a-1-x_{2})f_{12} \\
&+4(x_{1}+1)(x_{2}+1)(a-1-x_{1})(a-1-x_{2})(a-1-x_{1}-x_{2})(f_{11}f_{22}-f_{12}^{2}). 
\end{split}
\end{equation*}
 
The inverse is thus given by
$$(D^{2}u)^{-1} = \left(\begin{array}{cc} u^{11} & u^{12}\\u^{21} & u^{22}  \end{array}\right)= \left(\begin{array}{cc} \frac{A_{11}}{\mathcal{D}} & \frac{A_{12}}{\mathcal{D}}\\ \frac{A_{12}}{\mathcal{D}}& \frac{A_{22}}{\mathcal{D}} \end{array}\right), $$
where
\begin{equation*}
\begin{split}
A_{11} =&  2(x_{1}+1)(a-1-x_{1})((a^{2}-ax_{1}-x_{2}^{2}-2x_{2}-1)\\ &+ 2(x_{2}+1)(a-1-x_{2})(a-1-x_{1}-x_{2})f_{22}),\\
A_{12} =& -2(x_{1}+1)(x_{2}+1)(a-1-x_{1})(a-1-x_{2})(1+2(a-1-x_{1}-x_{2})f_{12}), \\ \text{ and } &\\
A_{22} =&  2(x_{2}+1)(a-1-x_{2})((a^{2}-ax_{2}-x_{1}^{2}-2x_{1}-1)\\ &+ 2(x_{1}+1)(a-1-x_{1})(a-1-x_{1}-x_{2})f_{11}).
\end{split}
\end{equation*}
Straightforward calculation yields the following.
\begin{lemma}\label{VertexCalc}
Let $g$ be a toric K\"ahler metric on $\mathbb{C}P^{2}\sharp 2\overline{\mathbb{C}P}^{2}$ of the form given by Equation (\ref{torKmet}). Then the inverse $u^{ij}$ satisfies
$$u^{ij}(-1,-1) = u^{ij}(-1,a-1) = u^{ij}(0,a-1) = 0$$
for $i,j\in \{1,2\}$. The derivatives satisfy
$$u^{11}_{1}(-1-1)=2, \ \ \ u^{12}_{1}(-1,-1)=u^{12}_{2}(-1,-1) = 0, \ \ \ u^{22}_{2}(-1,-1) = 2, $$
$$u^{11}_{1}(-1,a-1)=2, \ \ \ u^{12}_{1}(-1,a-1)=u^{12}_{2}(-1,a-1) = 0, \ \ \ u^{22}_{2}(-1,a-1) = -2, $$
and
$$u^{11}_{1}(0,a-1)=-2, \ \ \ u^{12}_{1}(0,a-1)=0, \ \ \ u^{12}_{2}(0,a-1) = 2, \ \ \ u^{22}_{2}(0,a-1) = -2. $$
\end{lemma}
The quasi-Einstein metrics on $\mathbb{C}P^{2}\sharp\overline{\mathbb{C}P}^{2}$ are Hermitian and have $J$-invariant Ricci tensor. If we search for metrics on $\mathbb{C}P^{2}\sharp 2\overline{\mathbb{C}P}^{2}$ that are invariant under the $\mathbb{T}^{2}\times \mathbb{Z}_{2}$ action described above and which are also $J$-invariant, we have the following result:

\begin{proposition}\label{CP2B2form}
Let $(g,\phi)$ be a Hermitian quasi-Einstein metric on $\mathbb{C}P^{2}\sharp 2\overline{\mathbb{C}P}^{2}$. Suppose $g$ is invariant under the action of $\mathbb{T}^{2}\times \mathbb{Z}_{2}$, has $J$-invariant Ricci tensor and $g=e^{2\sigma}g_{K}$ for a K\"ahler metric $g_{K}$. Then
$$\sigma = -\log(bt+c) \textrm{ and } \phi = m\log\left( \frac{dbt+dc+1}{bt+c} \right).$$
\end{proposition} 
\begin{proof}
The form of $\sigma$ follows from the fact that the Ricci tensor of $g$ is $J$-invariant. This forces the gradient with respect to $g_{K}$, $\nabla e^{-\sigma}$, to be a holomorphic vector field. Hence  $e^{-\sigma}$ is an affine function in the polytope coordinates, invariant under the $\mathbb{Z}_{2}$-action.  With respect to the metric $g$, ${\nabla^{2} \phi - \frac{1}{m} d\phi\otimes d\phi}$ is $J$-invariant. Calculation yields
$$\phi_{ij} - (\sigma'(\phi_{i}+\phi_{j})+\frac{1}{m}\phi_{i}\phi_{j})=0,$$
for $1\leq i,j\leq 2$. This equation can then be solved explicitly, yielding the result. 
\end{proof}
The Kim-Kim first integral (\ref{KK}) and boundary behaviour of the metric given in Lemma \ref{VertexCalc} give some constraints on the quantities $b,c,d$ and $\mu$.  
\begin{proposition}\label{phiconst}
Suppose $(g,\phi)$ is a quasi-Einstein metric on $\mathbb{C}P^{2}\sharp 2\overline{\mathbb{C}P}^{2}$ of the form given in Proposition \ref{CP2B2form}.
Then
$$ \frac{4b}{(c-2b)(dc+1-2db)} =\frac{1}{(c-2b)^2}-\frac{\mu}{(dc+1-2db)^2}, $$
$$0 = \frac{1}{(c+(a-2)b)^2}-\frac{\mu}{(dc+1+(a-2)db)^2}, $$
and
$$\frac{-2b}{(c+(a-1)b)(dc+1+(a-1)db)}=\frac{1}{(c+(a-1)b)^{2}}-\frac{\mu}{(dc+1+(a-1)db)^{2}} .$$
Moreover we have the following;
$$\int_{P} (e^{-\phi}-\mu e^{\left(\frac{2}{m}-1\right)\phi})e^{4\sigma}dx = 0. $$  
\end{proposition}
\begin{proof}
All the equations above can be derived by examining the Kim-Kim first integral (\ref{KK}).  Calculating quantities with respect to the K\"ahler metric $g_{K}$, this equation becomes
\begin{equation}\label{KKconf}
m(e^{2\sigma} - \mu e^{\frac{2\phi}{m}+2\sigma}) = \Delta \phi +2g_{K}(\nabla \sigma, \nabla \phi)-|\nabla \phi|^{2}.
\end{equation}
The first three equations now follow from Lemma \ref{VertexCalc}. For the integral constraint we note that the right-hand side of (\ref{KKconf}) can be written as $\Delta_{\mathcal{F}}\phi$ where 
$$\Delta_{\mathcal{F}}(\cdot):= \Delta(\cdot) -g_{K}(\nabla \mathcal{F},\nabla\cdot)$$
and 
$$ \mathcal{F} = \phi-2\sigma.$$
The result follows from noting that for any $\mathcal{F}$
$$\int_{M}\Delta_{\mathcal{F}}(\cdot) e^{-\mathcal{F}}dV_{g} = 0.$$
\end{proof}
If one fixes the value of $a$ (and so the particular K\"ahler class)  and the value of $m$, then Proposition \ref{phiconst} yields four equations in the four unknowns $b,c,d$ and $\mu$.  Given the result of Theorem \ref{mainT} it is sensible to look for conformally K\"ahler quasi-Einstein metrics on $\mathbb{C}P^{2}\sharp 2\overline{\mathbb{C}P}^{2}$ in the first Chern class $c_{1}$ which corresponds to setting $a=2$. Using a numerical program to evaluate the integral we find in the case $a=2$ and $m=2$  that the system of equations in Proposition \ref{phiconst} have the unique admissable solution (values are given to 6 significant figures)
$$ b \approx -0.0744357, \ \ \  c\approx 1.00482 , \ \ \ d \approx-0.463585 \textrm{ and } \mu \approx 0.282617.$$
Another use of the proposition is to rule out certain limiting behaviours of hypothetical families of quasi-Einstein metrics.  Suppose that a family of conformally K\"ahler quasi-Einstein metrics of the form given in  Proposition \ref{CP2B2form} converges smoothly to a conformally K\"ahler gradient Ricci soliton. Then, as well as the Wang-Zhu soliton, one could in theory also converge to the Chen-LeBrun-Weber metric. By the calculations in \cite{HM14} this would mean as $m\rightarrow \infty$,
$$b\rightarrow  -0.217907 \textrm{ and } c\rightarrow 1.000632.$$
As $\phi$ converges to a constant the integral constraint of Proposition \ref{phiconst}  would mean $\mu\rightarrow 1$.  One can then check that these values are not admissable as solutions of the contraints of Proposition \ref{phiconst} and so conclude that the Chen-LeBrun-Weber metric is not the smooth limit of such a hypothetical family.

\section{Future work}\label{Sect5}
In this section we list and comment on some future directions for research that our current work raises. 

\begin{enumerate}
\item Are there conformally K\"ahler analogues of the L\"u-Page-Pope metric on $\mathbb{C}P^{2}\sharp 2\overline{\mathbb{C}P}^{2}$? The second and fourth authors are currently investigating this question numerically using an algorithm they developed for numerically approximating toric K\"ahler metrics on $\mathbb{C}P^{2}\sharp 2 \overline{\mathbb{C}P}^{2}$ in \cite{HM14} and \cite{HM15}.
\item What is the significance of the conformally K\"ahler quasi-Einstein metrics? For example, in dimension 4, the K\"ahler metrics conformal to the Page and the Chen-LeBrun-Weber metrics are extremal K\"ahler metrics. An extremal K\"ahler metric is a critical point of the Calabi energy and such metrics are the subject of intense research activity at the time of writing. The K\"ahler classes of the extremal metrics conformal to the Page and Chen-LeBrun-Weber Einstein metrics are distinguished by minimising the Calabi energy over all possible K\"ahler classes on these manifolds \cite{LB}.
\item The existence theorem for compact quasi-Einstein metrics in \cite{Hallqem} can be paraphrased as: ``whenever the manifold $W_{q}$ admits a non-trivial K\"ahler-Ricci soliton, it admits at least one family of quasi-Einstein metrics.'' Is this true in general? Or is it at least true for other constructions of K\"ahler-Ricci solitons, such as that of Wang-Zhu \cite{WZ}? 
\end{enumerate}

\end{document}